\documentclass[12pt]{article}
\usepackage{fullpage,
}

\usepackage{bm}
\usepackage{amsmath,amssymb}
\usepackage{amsthm}

\usepackage{epsfig}
\usepackage{graphicx}
\usepackage{makeidx}

\newcommand{\K}{{\mathcal K}}

\newcommand{\N}{{\cal N}}

\newtheorem{thm}{Theorem}
\newtheorem{lem}[thm]{Lemma}

\newcommand{\bk}{{\rm bk}}

\newcommand{\true}{{\rm true}}

\graphicspath{{../../figures/}}
\begin{document}

\title{Functional Gaussian processes for regression with linear PDE models}
\date{\empty}

\author{ N.~C. Nguyen\footnote{MIT Department of Aeronautics and Astronautics, 77 Massachusetts Ave., Cambridge, MA 02139, USA, email: \texttt{cuongng@mit.edu}} \, and J. Peraire\footnote{MIT Department of Aeronautics and Astronautics, 77 Massachusetts Ave., Cambridge, MA 02139, USA, email: \texttt{peraire@mit.edu}}}

\maketitle

\begin{abstract}
In this paper, we present a new statistical approach to the problem of incorporating experimental observations into a mathematical model described by linear partial differential equations (PDEs) to improve the prediction of the state of a physical system. We augment the linear PDE with a functional that accounts for the uncertainty in the mathematical model and is modeled as a {\em Gaussian process}. This gives rise to a stochastic PDE which is characterized by the Gaussian functional. We develop a {\em functional Gaussian process regression}  method to determine the posterior mean and covariance of the Gaussian functional, thereby solving the stochastic PDE to obtain the posterior distribution for our prediction of the physical state. Our method has the following features which distinguish itself from other regression methods. First, it incorporates both the mathematical model and the observations into the regression procedure. Second, it can handle the observations given in the form of linear functionals of the field variable.  Third, the method is non-parametric in the sense that it provides a systematic way to optimally determine the prior covariance operator of the Gaussian functional based on the observations. Fourth, it provides the posterior distribution quantifying the magnitude of uncertainty in our prediction of the physical state.  We present numerical results to illustrate these features of the method and compare its performance to that of the standard Gaussian process regression.
\end{abstract}

\newpage

\section{Introduction}

 
Partial differential equations (PDEs) are used to mathematically model a wide variety of physical phenomena such as heat transfer, fluid flows, electromagnetism, and structural deformations. A PDE model of a physical system is typically described by conservation laws, constitutive laws, material properties, boundary conditions, boundary data, and geometry. In practical applications, the mathematical model described by the PDEs is only an approximation to the real physical system due to ({\em i}) the deliberate simplification of the mathematical model to keep it tractable (by ignoring certain physics or certain boundary conditions that pose computational difficulties), and ({\em ii}) the uncertainty of the available data (by using geometry, material property and boundary data that are not exactly the same as those of the physical system). We refer to the PDE model (available to us) as the {\em best knowledge PDE model} \cite{Yano2013} and to its solution as the {\em best knowledge state}. To assess the accuracy of the best knowledge model in predicting the physical system, the best knowledge state needs to be compared against experimental data, which typically will have some level of noise.

In cases where the discrepancy between the PDE model and the experimental data is beyond an acceptable level of accuracy, we need to improve the current PDE model. There are  several approaches to defining a new improved model. {\em Parameter estimation} \cite{Tarantola2005,Zhang1997} involves calibrating  some parameters in the model to match the data.  An alternative approach to obtain an improved model is {\em data assimilation} \cite{Kalman1960,LewisLakshmivarahanDhall2006,Li2001,Lorenc1981,Sondergaard2013}. Broadly speaking, data assimilation is a numerical procedure by which we incorporate observations into a mathematical model to reflect the errors inherent in our mathematical modeling of the physical system. Although data assimilation shares the same objective as parameter estimation, it differs from the latter in methodology. More specifically, data assimilation does not assume any parameters to be calibrated; instead, data assimilation defines a new model that matches the observations as well as possible, while being as close as possible to the best knowledge model. Another approach is {\em data interpolation} \cite{Barrault2004,everson95karhunenloeve,Maday2008,Nguyen2008a,Patera2012,willcox06:_gappy} which involves computing  a collection of solutions (snapshots) of a parametrized or time-varying mathematical model and reconstructing the physical state by fitting the experimental data to the snapshots.

A widely used technique for obtaining an improved model in parameter estimation and data assimilation is {\em least squares regression} \cite{Cohen2013,LewisLakshmivarahanDhall2006,JohnWolbergAu2005}. Least squares is a deterministic regression approach that provides an estimate for the physical state which is optimal in least squares sense. However, it does not provide a means to quantify the prediction uncertainty. A recent work \cite{Yano2013} poses the least-square regression as a regularized saddle point Galerkin formulation which admits interpretation from a variational framework and permits its extension to {\em Petrov-Galerkin formulation}. While the Petrov-Galerkin formulation provides more flexibility than the Galerkin formulation, it does not quantify the uncertainty in the prediction either. A popular statistical approach in parameter estimation and data assimilation is {\em Bayesian inference} \cite{ElMoselhy2012,AndrewGelmanJohnCarlinHalSternDavidDunsonAkiVehtari2013,SimoSarkka2013,Wikle2007}. In Bayesian inference an estimate of the physical state is described by random variables and the posterior probability distribution of the estimate is determined by the data according to Bayes' rule \cite{AndrewGelmanJohnCarlinHalSternDavidDunsonAkiVehtari2013,SimoSarkka2013}.  Therefore, Bayesian inference provides a powerful framework to quantify the prediction uncertainties.

In this paper, we introduce a new statistical data assimilation approach to the problem of incorporating observations into the best knowledge model to predict the state of physical system. Our approach has its root in Gaussian process (GP) regression \cite{O'Hagan1978,RasWil06,Sacks1989}. We augment the linear PDE with a functional that accounts for the uncertainty in the mathematical model and is modeled as a {\em Gaussian process}.\footnote{In the cases considered, the physical system is not stochastic but deterministic. The introduction of the Gaussian functional serves to represent uncertainties in the best knowledge model and in the data, {\em not} in the physical system {\em per se}.} This gives rise to a stochastic PDE whose solutions are characterized by the Gaussian functional. By extending the standard GP regression for {\em functions of vectors} to {\em functionals of functions}, we develop a {\em functional Gaussian process regression} method  to determine the posterior distribution of the Gaussian functional, thereby solving the stochastic PDE for our prediction of the physical state. Our method is devised as follows. We first derive a {\em functional regression problem} by making use of the adjoint states and the observations. We next solve the functional regression problem by an application of the principle of Gaussian processes to obtain the posterior mean and covariance of the Gaussian functional. Finally, we compute the posterior distribution for our estimate of the physical state. A crucial ingredient in our method is the {\em covariance operator} representing the {\em prior} of the Gaussian functional. The bilinear covariance operators considered incorporate a number of free parameters (the so-called {\em hyperparameters}) that can be optimally determined from the measured data by maximizing a marginal likelihood.

Our functional Gaussian process regression method can be viewed as a generalization of the standard GP regression from a finite dimensional vector (input) space to an infinite dimensional function (input) space. GP regression is a well-established technique to construct maps between inputs and outputs based on a set of sample, or training, input and output pairs, but does not offer a direct method to incorporate prior knowledge, albeit approximate, from an existing model relating the inputs and outputs. By combining the best knowledge model with the data, our method can greatly improve the prediction of the physical system. 

Furthermore, we introduce a {\em nonparametric Bayesian inference} method for linear functional regression with Gaussian noise. It turns out that nonparametric Bayesian inference and functional GP regression represent two different views of the same procedure. Specifically, we can think of functional GP regression as defining a distribution over functionals and doing inference in the space of functionals --- the {\em functional-space view}. We can think of nonparametric Bayesian inference as defining a distribution over weights and doing inference in the space of weights --- the {\em weight-space view}. Theoretically, functional GP regression can be interpreted as an application of the {\em kernel trick} \cite{RasWil06} to nonparametric Bayesian inference, thereby avoiding an explicit construction of the feature map.

The paper is organized as follows. In Section 2, we present a description of the problem considered.  In Section 3, we give an overview of Gaussian processes for regression problems. In Section 4, we introduce functional Gaussian processes for functional regression problems via linear PDE models. In Section 5, we present numerical results to demonstrate our method and compare its performance to that of function Gaussian process regression. In Section 6, we provides some concluding remarks on future research. Finally, in the Appendix, we describe our nonparametric Bayesian inference method.

\section{Motivation and Problem Statement}

Let $\Omega^\true \in \mathbb{R}^n$ denote a bounded open domain with Lipschitz boundary. Let $V^\true(\Omega^\true)$ be an appropriate real-valued function space in which the true state $u^{\true}$ resides. A weak formulation of the true PDE model can be stated as: Find $u^{\true} \in V^{\true}$ and $\bm s^{\true} \in \mathbb{R}^M$ such that
\begin{subequations}
\label{eq5}
\begin{alignat}{2}
a^{\true}(u^{\true}, v) & = \ell^{\true}(v),  \quad \quad \ \forall\, v \in V^{\true},\\
s^{\true}_i  & =  c^{\true}_i(u^{\true}), \quad i = 1,\ldots,M,
\end{alignat}
\end{subequations}
where $a^{\true} : V^{\true} \times V^{\true} \to \mathbb{R}$ is a bilinear form, $\ell^{\true} : V^{\true} \to \mathbb{R}$ is a linear functional, and $c_i^{\true}  : V^{\true} \to \mathbb{R}, i = 1,\ldots,M$ are observation functionals. We assume that the true PDE model (\ref{eq5}) is well defined and accurately describes the physical system of interest. 

In actual practice, we do not have access to $a^{\true}, \ell^{\true},$ $c_i^{\true}$, and $V^\true(\Omega^\true)$. Hence, we can not compute $u^{\true}$ and $\bm s^{\true}$. However, we assume that we have access to the ``best knowledge'' of $a^{\true}, \ell^{\true},$ $c_i^{\true}, i = 1,\ldots, M$, and $V^\true(\Omega^\true)$, which shall be denoted by $a^{\bk}, \ell^{\bk},$ $c_i^{\bk}, i = 1,\ldots, M$, and $V^\bk(\Omega^\bk)$, respectively. We then define the best knowledge PDE model:  Find $u^{\rm bk} \in V^\bk$ and $\bm s^{\rm bk} \in \mathbb{R}^M$ such that
\begin{subequations}
\label{eq6}
\begin{alignat}{2}
a^{\rm bk}(u^{\rm bk}, v) & = \ell^{\rm bk}(v),  \quad \ \ \forall\,  v \in V^\bk,\\
s_i^{\rm bk}  & =  c^{\rm bk}_i(u^{\rm bk}), \quad i = 1,\ldots,M .
\end{alignat}
\end{subequations}
In the remainder of this paper, we shall drop the superscript ``bk'' for the quantities associated with the best knowledge model  to simplify the notation. (In practice, we replace the continuous function space $V(\Omega)$ with a discrete approximation space, which is assumed to be large enough that the discrete solution is indistinguishable from the continuous one.)

We now assume that we are given the observed data $\bm d \in \mathbb{R}^M$, which are the $M$ measurements of the true output vector $\bm s^{\true}$. We further assume that the measurements differ from the true outputs $\bm s^{\rm true}$ by additive Gaussian noise $\bm \varepsilon$, namely,
\begin{equation}
\label{datatrue}
\bm d = \bm s^{\true} + \bm \varepsilon,
\end{equation}
where $\varepsilon_i, i = 1,\ldots,M$ are independent, identically distributed Gaussian distributions with zero mean and variance $\sigma^2$.  If $\sigma$ is sufficiently small within the acceptable accuracy then we can use the observed data $\bm d$ to validate the best knowledge model (\ref{eq6}). If the best knowledge outputs $\bm s$ are close enough to $\bm d$ within the noise level then we may trust the best knowledge model to predict the behavior of the true model. In many cases, the best knowledge outputs do not match the observed data due to various sources of uncertainty from physical modeling, constitutive laws, boundary conditions, boundary data, material properties, and geometry. 

We are interested in improving the best knowledge model when it does not produce a good estimate of the true state. In particular, we propose a method to compute a better estimate for the true state by combining the best knowledge model with the observed data. Our method has its root in Gaussian process regression. Before proceeding to describe the proposed method we review the ideas behind Gaussian processes.

\section{Gaussian Process Regression}

We begin by assuming that we are given a training set of $M$ observations
\begin{equation}
 S = \{(\bm x_i,y_i), i = 1,\ldots,M\}, 
 \end{equation}
where $\bm x_i \in \mathbb{R}^N$ denotes an input vector of dimension $N$ and $y_i \in \mathbb{R}$ denotes a scalar real-valued output. The training input vectors $\bm x_i, i=1,\ldots,M$ are aggregated in the $N \times M$ real-valued matrix $\bm X$, and the outputs are collected in the real-valued vector $\bm y$, so we can write $S = (\bm X,\bm y)$. We assume that 
\begin{equation}
\label{inputoutput}
y_i = h(\bm x_i) + \mathcal{N}(0,\sigma^2), 
\end{equation}
where $h(\bm x)$ is the true but unknown function which we want to infer. The unknown function $h(\bm x)$ is modeled as a Gaussian process\footnote{A Gaussian process is a generalization of the Gaussian probability distribution. A Gaussian process governs the properties of Gaussian random functions, whereas a Gaussian probability distribution describes the properties of Gaussian random variables (scalars or vectors).} with zero mean\footnote{The Gaussian process is assumed to have zero mean because we can always subtract the original outputs $\bm y$ from its average $\bar{y} = \frac{1}{M} \sum_{i=1}^M y_i$ to obtain new outputs with zero average. We then work with the new outputs and add the average $\bar{y}$ to our Gaussian process estimator.}, for simplicity,  and covariance function $\kappa(\bm x,\bm x')$, namely, 
\begin{equation}
\label{hprior}
h(\bm x) \sim \mathcal{GP}(0, \kappa(\bm x,\bm x')) \ .
\end{equation}
From a Bayesian perspective, we encode our belief that instances of $h(\bm x)$ are drawn from a Gaussian process with zero mean and covariance function $\kappa(\bm x, \bm x')$ {\em prior to taking into account observations}.\footnote{At this point, one may ask what if our belief is wrong, that is, what if the covariance function $\kappa$ is not correctly chosen. Of course, choosing a wrong covariance function will result in very poor prediction. Hence, the covariance function should not be chosen arbitrarily. As discussed later, Gaussian processes provide a framework for optimal selection of a covariance function based on the observed data.} Mathematically speaking, $h(\bm x)$ is assumed to reside in a reproducing kernel Hilbert space \cite{Aronszajn} spanned by the eigenfunctions of the covariance function $\kappa(\bm x, \bm x')$.

Let $\bm X^*$ be the $N \times M^*$ matrix that contains $M^*$ test input vectors $\bm x^*_j, j =1,\ldots,M^*$ as its columns. Since $h(\bm x) \sim \mathcal{GP}(0, \kappa(\bm x,\bm x'))$, the joint distribution of the observed outputs and the function values at the test input vectors  is
\begin{equation}
\label{jointdistribute}
\left[
\begin{array}{c}
\bm y \\
\bm h^*
\end{array}
\right] \sim \mathcal{N}\left(0, 
\left[
\begin{array}{cc}
\bm{\K}(\bm{X},\bm{X}) + \sigma^2 \bm I & \bm{\K}(\bm{X},\bm{X}^*) \\
\bm{\K}(\bm{X}^*,\bm{X}) & \bm{\K}(\bm{X}^*,\bm{X}^*)
\end{array}
\right]
 \right),
\end{equation}
where $\bm h^* \in \mathbb{R}^M$, $\bm{\K}(\bm{X},\bm{X}) \in \mathbb{R}^{M \times M}$, $\bm{\K}(\bm{X},\bm{X}^*) \in \mathbb{R}^{M \times M^*}$, $\bm{\K}(\bm{X}^*,\bm{X}) \in \mathbb{R}^{M^* \times M}$, and $\bm{\K}(\bm{X}^*,\bm{X}^*) \in \mathbb{R}^{M^* \times M^*}$ have entries 
\begin{equation}
\begin{array}{rcll}
h^*_i  & = & h(\bm x^*_i), &  \quad i = 1, \ldots, M^* ,  \\
{\K}_{ij}(\bm{X},\bm{X}) & = & \kappa(\bm{x}_i,\bm{x}_j), &  \quad i = 1, \ldots, M, j = 1, \ldots, M, \\
{\K}_{ij}(\bm{X},\bm{X}^*) & = & \kappa(\bm{x}_i,\bm x^*_j), &  \quad i = 1, \ldots, M, j = 1, \ldots, M^* , \\
{\K}_{ij}(\bm{X}^*,\bm{X}) & = & \kappa(\bm x^*_i,\bm{x}_j), &  \quad i = 1, \ldots, M^*, j = 1, \ldots, M , \\
{\K}_{ij}(\bm{X}^*,\bm{X}^*) & = & \kappa(\bm{x}_i^*,\bm{x}_j^*), &  \quad i = 1, \ldots, M^*, j = 1, \ldots, M^* ,
\end{array}
\end{equation} 
respectively. We next apply the conditional distribution formula (see \cite{RasWil06}) to the joint distribution (\ref{jointdistribute}) to obtain the {\em predictive distribution}  for $\bm h^*$ as
\begin{equation}
\bm h^* |\bm  y, \bm{X}, \bm{X}^* \sim \mathcal{N}(\bar{\bm h}^*, \mathrm{cov}(\bm h^*)),
\end{equation}
where
\begin{equation}
\begin{split}
\bar{\bm h}^* &= \bm{\K}(\bm{X}^*,\bm{X}) \bm \alpha, \\
\mathrm{cov}(\bm h^*) &= \bm{\K}(\bm{X}^*,\bm{X}^*) - \bm{\K}(\bm{X}^*,\bm{X}) \bm C^{-1} \bm{\K}(\bm{X},\bm{X}^*) ,
\end{split}
\end{equation}
and $\bm \alpha \in \mathbb{R}^M$ and $\bm C \in \mathbb{R}^{M \times M}$ are given by
\begin{equation}
\bm C  \bm \alpha = \bm y, \qquad \bm C = \bm{\K}(\bm{X},\bm{X}) + \sigma^2 \bm I  \ .
\end{equation}
Note that the predictive mean $\bar{\bm h}^*$  is a linear combination of $M$ kernel functions, each one centered on a training input vector. Note also that the predictive covariance $\mathrm{cov}(\bm h^*)$ does not explicitly depend on the observed data $\bm y$, but only on the training input vectors $\bm X$ and the covariance function $\kappa$. This is a property of the Gaussian distribution. However, as discussed below, the covariance function $\kappa$ can be determined by using the observed data. As a result, the predictive covariance implicitly depends on the observed data.

Typically, the covariance function $\kappa$ has some free parameters $\bm \theta = (\theta_1,\ldots,\theta_Q)$, so that the matrix $\bm C$ depends on $\bm \theta$. These free parameters are called {\em hyperparameters}. The hyperparameters have a significant impact on the predictive mean and covariance.  They are determined by maximizing the log marginal likelihood (see \cite{RasWil06}):
\begin{equation}
\label{LMLGP}
\log p(\bm y | \bm{X}, \bm \theta) = - \frac{1}{2} \bm y^T \bm C(\bm \theta)^{-1} \bm y - \frac{1}{2} \log (\det(\bm C(\bm \theta))) - \frac{M}{2} \log (2 \pi) .
\end{equation}
Once we choose a specific form for $\kappa$ and determine its hyperparameters, we can compute $\bar{\bm h}^*$ and $\mathrm{cov}(\bm h^*)$ for any given $\bm X^*$. Gaussian processes also provide us a mean to choose an appropriate family among many possible families of covariance functions. Choosing a covariance function for a particular application involves both determining hyperparameters within a family and comparing across different families. This step is termed as {\em model selection} \cite{RasWil06}. 

We see that the standard GP regression provides us not only a posterior mean, but also a posterior covariance which characterizes uncertainty in our prediction of the true function. Moreover, it allows us to determine the optimal covariance function and thus the optimal reproducing Kernel Hilbert space in which the true function is believed to reside. These features differentiate GP regression from parametric regression methods such as least-squares regression, which typically provides the maximum likelihood estimate only.  However,  GP regression tends to require larger sample sizes than parametric regression methods because the data must supply enough information to yield a good covariance function by using model selection. 

There are a number of obstacles that prevent us from applying the standard GP regression to our problem of interest described in the previous section. First, our outputs are in general not the evaluations of the state at spatial coordinates. Instead, they are linear functionals of the state. Second, the standard GP regression described here does not allow us to make use of the best knowledge model. The best knowledge model plays an important role because it carries crucial prior information about the true model. By taking advantage of the best knowledge model, we may be able to use far less observations to obtain a good prediction and thus address the main disadvantage of the standard GP regression. We propose a new approach to overcome these obstacles.

\section{Functional Gaussian Process Regression}

\subsection{A stochastic PDE model}

Let  $g : V \to \mathbb{R}$ be a linear functional. We introduce a new mathematical model: Find $u^* \in V$ and $\bm s^* \in \mathbb{R}^M$ such that
\begin{subequations}
\label{eq7}
\begin{alignat}{2}
a(u^*, v)  + g(v) & = \ell(v) ,  \quad \ \ \forall v \in V,\\
s^*_i  & =  c_i(u^*), \quad i = 1,\ldots,M .
\end{alignat}
\end{subequations}
Notice that the new model (\ref{eq7}) differs from the best knowledge model (\ref{eq6}) by the functional $g$. We can determine $u^*$ and $\bm s^*$ only if $g$ is known. The functional $g$ thus characterizes the solution $u^*$ and the output vector $\bm s^*$ of the model (\ref{eq7}). We note that if $g(v) = \ell(v) - a(u^{\true},v)$ then we $u^* = u^{\true}$. Unfortunately, this particular choice of  $g$ requires the true state $u^\true$ which we do not know and thus want to infer. 

In order to capture various sources of uncertainty in the best knowledge model,  we represent  $g$ as a {\em functional Gaussian process}\footnote{For now, a functional Gaussian process can be thought of as a generalization of the Gaussian process from a finite dimensional vector space to an infinite dimensional function space. } with  zero mean and covariance  operator $k$, namely,
\begin{equation}
\label{fprior}
g(v) \sim \mathcal{FGP}(0, k(v,v')), \qquad \forall\, v, v' \in V\ .
\end{equation}
Notice that there are three main differences between the functional Gaussian process (\ref{fprior}) and the Gaussian process (\ref{hprior}).  First, $g$ is a functional, whereas $h$ is a function. Second, $v$ is a function, whereas $\bm x$ is a vector. And third, $k(v,v')$ is generally a differential and integral operator of $v$ and $v'$, whereas $\kappa(\bm x, \bm x')$ is a function of $\bm x$ and $\bm x'$.  We will require that the covariance operator $k : V \times V \to \mathbb{R}$ is symmetric positive-definite. That is,
\begin{equation}
k(v,v') = k(v',v),  \quad \mbox{and}  \quad k(v,v) > 0, \qquad \forall v, v'  \in V \ . 
\end{equation}
As the covariance operator $k$ characterizes the space of all possible functionals prior to taking into account the observations, it plays an important role in our method. The selection of a covariance operator will be discussed later.

Since $g$ is a functional Gaussian process, the model (\ref{eq7}) becomes a stochastic PDE.  In order to solve the stochastic PDE (\ref{eq7}), we need to compute the {\em posterior mean} and {\em posterior covariance} of $g$ after accounting for the observed data $\bm d$. To this end, we formulate a functional regression problem and describe a procedure for solving it as follows.

\subsection{Functional regression problem}

We first introduce the adjoint problems: for $i = 1,\ldots,M$ we find $\phi_i \in V$ such that
\begin{equation}
\label{adjoint}
a(v,\phi_i) = -c_i(v), \quad \forall v \in V.
\end{equation}
We note that the adjoint states $\phi_i$ depend only on the output functionals $c_i$ and the bilinear form $a$. It follows from (\ref{eq6}), (\ref{eq7}), and (\ref{adjoint}) that
\begin{equation}
\label{adjointout}
g(\phi_i)  = \ell(\phi_i) - a(u^*, \phi_i) = a(u,\phi_i) + c_i(u^*) = c_i(u^*) -  c_i(u) = s^*_i - s_i ,
\end{equation}
for $i = 1,\ldots,M$. Moreover, we would like our stochastic PDE to produce the outputs $\bm s^*$ that are consistent with the observed data $\bm d$ in such a way that
\begin{equation}
\label{dataout}
 d_i = s^*_i  +  \N(0,\sigma^2), \qquad i = 1,\ldots,M. 
\end{equation}
This equation is analogous to (\ref{datatrue}) which relates the observed data $\bm d$ to the true outputs $\bm s^{\rm true}$. We substitute $s_i^* = d_i  -  \N(0,\sigma^2)$ into (\ref{adjointout}) to obtain
\begin{equation}
d_i  - s_i =  g(\phi_i) +  \N(0,\sigma^2) , \qquad i = 1,\ldots,M .
\end{equation}
Notice that this expression characterizes the relationship between $g(\phi_i)$ and $d_i - s_i$ in the same way (\ref{inputoutput}) characterizes the relationship between $h(\bm x_i)$ and $y_i$.

We now introduce a training set of $M$ observations
\begin{equation}
T = \{(\phi_i,d_i - s_i), \ i = 1,\ldots,M\} ,
 \end{equation}
and use this training set to learn about $g$. More specifically, we wish to determine $g(\phi^*)$ for any given $\phi^* \in V$ based on the training set $T$. This problem is similar to the regression problem described in the previous section and is named the {\em functional regression problem} to emphasize that the object of interest $g$ is a functional. We next describe the solution of the functional regression problem.

\subsection{Regression procedure}

Let $\Phi = [\phi_1,\ldots,\phi_M]$ be a collection of $M$ adjoint states as determined by (\ref{adjoint}). Let $\Phi^* = [\phi^*_1 \in V,\ldots,\phi_M^* \in V]$ be a collection of $M^*$ test functions. The joint distribution of the observed outputs and the functional values for the test functions according to the prior (\ref{fprior}) is given by
\begin{equation}
\label{jointdist}
\left[
\begin{array}{c}
\bm d - \bm s \\
\bm g^*
\end{array}
\right] \sim \mathcal{N}\left(0, 
\left[
\begin{array}{cc}
\bm{K}(\Phi,\Phi) + \sigma^2  \bm I & \bm{K}(\Phi,\Phi^*) \\
\bm{K}(\Phi^*,\Phi) & \bm{K}(\Phi^*,\Phi^*)
\end{array}
\right]
 \right),
\end{equation}
where $\bm g^* \in \mathbb{R}^{M^*}$, $\bm{K}(\Phi,\Phi) \in \mathbb{R}^{M \times M}$, $\bm{K}(\Phi,\Phi^*) \in \mathbb{R}^{M \times M^*}$, $\bm{K}(\Phi^*,\Phi) \in \mathbb{R}^{M^* \times M}$, and $\bm{K}(\Phi^*,\Phi^*) \in \mathbb{R}^{M^* \times M^*}$ have entries 
\begin{equation}
\begin{array}{rcll}
g^*_i  & = & g(\phi^*_i), &  \quad i = 1, \ldots, M^* ,  \\
{K}_{ij}(\Phi,\Phi) & = & k(\phi_i,\phi_j), &  \quad i = 1, \ldots, M, j = 1, \ldots, M, \\
{K}_{ij}(\Phi,\Phi^*) & = & k(\phi_i,\phi^*_j), &  \quad i = 1, \ldots, M, j = 1, \ldots, M^*, \\
{K}_{ij}(\Phi^*,\Phi) & = & k(\phi^*_i,\phi_j), &  \quad i = 1, \ldots, M^*, j = 1, \ldots, M , \\
{K}_{ij}(\Phi^*,\Phi^*) & = & k(\phi_i^*,\phi_j^*), &  \quad i = 1, \ldots, M^*, j = 1, \ldots, M^* ,
\end{array}
\end{equation} 
respectively. It thus follows that the {\em predictive distribution}  for $\bm g^*$ is 
\begin{equation}
\bm g^* | (\bm d - \bm s), \Phi, \Phi^* \sim \mathcal{N}(\bar{\bm  g}^*, \mathrm{cov}(\bm  g^*)),
\end{equation}
where
\begin{equation}
\bar{\bm  g}^* = \bm{K}(\Phi^*,\Phi) \bm  \beta,  \qquad \mathrm{cov}(\bm  g^*) = \bm{K}(\Phi^*,\Phi^*) - \bm{K}(\Phi^*,\Phi) \bm  D^{-1} \bm{K}(\Phi,\Phi^*) ,
\end{equation}
and $\bm  \beta \in \mathbb{R}^M$ and $\bm  D \in \mathbb{R}^{M \times M}$ are given by
\begin{equation}
\label{betaD}
\bm  D \bm  \beta = \bm  d - \bm  s, \qquad \bm  D = \bm{K}(\Phi,\Phi) + \sigma^2 \bm  I  \ .
\end{equation}
Notice that we have correspondence with function Gaussian process regression described in the previous section, when identifying $(\Phi,\bm d - \bm s)$ with $(\bm X,\bm  y)$, $(\Phi^*,\bm g^*)$ with $(\bm X^*,\bm  h^*)$, and $k(\cdot,\cdot)$ with $\kappa(\cdot,\cdot)$. 

While our approach share similarities with function Gaussian process regression, it differs from the latter in many important ways. We summarize in Table 1 the differences between function Gaussian process regression and functional Gaussian process regression. 

\begin{table}[ht]
\label{tab1}
\begin{center}
\scalebox{0.915}{
\begin{tabular}{|c||c|c|}
\hline
Quantities & Function Gaussian process & Functional Gaussian process \\
\hline 
\hline
& &   \\[-2ex]
Input &vector $\bm x \in \mathbb{R}^N$ & function $v \in V$ \\[1ex]
\hline
& & \\[-2ex]
Output &function $h(\bm x)$ & functional $g(v)$ \\[1ex]
\hline
& & \\[-2ex]
Prior & $h \sim \mathcal{N}(0,\kappa(\cdot,\cdot))$ & $g \sim \mathcal{N}(0,k(\cdot,\cdot))$ \\[1ex]
\hline
& & \\[-2ex]
Kernel &  function $\kappa : \mathbb{R}^N \times \mathbb{R}^N \to \mathbb{R}$ & operator $k: V \times V \to \mathbb{R}$ \\[1ex]
\hline
& & \\[-2ex]
Training inputs & $\bm X \in \mathbb{R}^{N \times M}$ & $\Phi \in V^M$ adjoint states \\[1ex]
\hline
& & \\[-2ex]
Observations & $\bm y = h(\bm X) + \mathcal{N}(0,\sigma)$ &  $\bm d - \bm s = g(\Phi)  + \mathcal{N}(0,\sigma)$  \\[1ex]
\hline
& & \\[-2ex]
Coefficients & $\bm C \bm  \alpha = \bm y,  \bm C = \left[ \bm{\K}(\bm X,\bm X) + \sigma^2\bm  I \right]$ & $\bm D  \bm \beta = \bm d - \bm s,\bm  D = \left[ \bm K(\Phi,\Phi) + \sigma^2 \bm I \right]$  \\[1ex]
\hline
& & \\[-2ex]
Test inputs & $\bm X^* \in \mathbb{R}^{N \times M^*}$ & $\Phi^* \in V^{M^*}$ \\[1ex]
\hline
& & \\[-2ex]
Mean & $\bm{\K}(\bm X^*,\bm X)\bm  \alpha$ &  $\bm K(\Phi^*,\Phi) \bm \beta$ \\[1ex]
\hline
& & \\[-2ex]
Covariance & $\bm{\K}(\bm X^*\bm ,\bm X^*) - \bm{\K}(\bm X^*,\bm X) \bm C^{-1} \bm{\K}(\bm X,\bm X^*)$ &  $\bm K(\Phi^*,\Phi^*) - \bm K(\Phi^*,\Phi)\bm  D^{-1} \bm K(\Phi,\Phi^*) $ \\[1ex]
\hline
\end{tabular}
}
\caption{Comparison between function Gaussian process regression and functional Gaussian process regression. Note that the best knowledge model enters in the functional Gaussian process regression through the adjoint states $\Phi$ and the best knowledge outputs $\bm s$.}
\end{center}
\end{table}

In the Appendix A, we introduce a nonparametric Bayesian framework for linear functional regression with Gaussian noise. It turns out that this nonparametric Bayesian framework is equivalent to the functional GP regression described here. In fact,  functional GP regression can be viewed as an application of the {\em kernel trick} to nonparametric Bayesian inference for linear functional regression, thereby avoiding the computation of the eigenfunctions of the covariance operator $k$. We next introduce a family of bilinear covariance operators and then describe a method for determining the hyperparameters.

\subsection{Covariance operators}

The covariance operator $k$ is a crucial ingredient in our approach. Here, we consider a class of bilinear covariance operators parametrized by $\bm \theta = (\theta_1, \theta_2)$  of the form:
\begin{equation}
k(v,v';\bm \theta) = \theta_{1} \int_\Omega v v' d \bm x +  \theta_{2}  \int_{\Omega} \nabla v \cdot \nabla v' d \bm x \ .
\end{equation}
More general forms of the covariance operator are possible provided that they are symmetric and positive definite.

In order for a covariance operator to be used in our method, we need to specify its hyperparameters $\bm \theta$. Fortunately, Gaussian processes allow us to determine the hyperparameters by using the observed data. In order to do this, we first calculate the probability of the observed data given the hyperparameters, or {\em marginal likelihood} and choose $\bm \theta$ so that this likelihood is maximized. We note from (\ref{jointdist}) that
\begin{equation}
\label{ML}
p(\bm d - \bm s |\Phi, \bm \theta ) = \N(0, \bm D(\bm \theta)),
\end{equation}
where the matrix $\bm D(\bm \theta)$ as defined in (\ref{betaD}) depends on $k(\cdot,\cdot;\bm \theta)$ and thus on $\bm \theta$ as well. Rather than maximizing (\ref{ML}), it is more convenient to maximize the log marginal likelihood which is given by,
\begin{equation}
\label{LML}
\log p(\bm d - \bm s |\Phi, \bm \theta ) =  - \frac{1}{2} (\bm d- \bm s)^T \bm D(\bm \theta)^{-1} (\bm d- \bm s) - \frac{1}{2} \log ( \det(\bm D(\bm \theta)))- \frac{M}{2} \log (2 \pi) .
\end{equation}
Thus, we find $\bm \theta$ by solving the maximization problem 
\begin{equation}
\bm \theta = \arg \max_{\bm \theta'  \in \mathbb{R}^{2}} \log p(\bm d - \bm s |\Phi, \bm \theta' ) .
\end{equation}
Hence, the hyperparameters $\bm \theta$ are chosen as the maximizer of the log marginal likelihood. 


Once we determine the covariance operator,  we can compute $(\bar{\bm  g}^*, \mathrm{cov}(\bm  g^*))$ for any given set $\Phi^*$ of test functions as described in Subsection 4.3. Therefore, functional GP regression is {\em non-parametric} in the sense that both the  hyperparameters are chosen in light of the observed data. In order words, the data are used to define both the prior covariance and the posterior covariance. In contrast, parametric regression methods use a number of parameters to define the prior and combine this prior with the data to determine the posterior prediction. It remains to describe how to compute the posterior mean and covariance of the solution  $u^*$ of the stochastic PDE.

\subsection{Computation of the mean state and covariance}

We recall that our stochastic PDE model consists of finding ${u}^* \in V$ such that
\begin{equation}
\label{weakform}
a({u}^*, v)    = \ell(v) - g(v),  \quad \forall v \in V .
\end{equation}
Let $\{v_j(\bm x)\}_{j = 1}^{J}$ be a ``suitable'' basis set of the function space $V(\Omega)$, where $J$ is the dimension of  $V(\Omega)$. Since the functional $g$ is Gaussian and the best knowledge model is linear, we can express the solution of the stochastic PDE (\ref{weakform}) as  
\begin{equation}
\label{ustarsum}
u^*(\bm x) = \sum_{j=1}^J \gamma_j^* v_j(\bm x),  \qquad \bm {\gamma}^* \sim \N(\bar{{\bm \gamma}}^*,\mathrm{cov}({\bm \gamma^*})) .
\end{equation}  
In order to determine $\bar{{\bm{\gamma}}}^*$ and $\mathrm{cov}({\bm{\gamma}^*})$, we choose $v = v_i, i = 1,\ldots,J$ in (\ref{weakform}) to arrive at the stochastic linear system:
\begin{equation}
\label{linearsystem}
\bm {A} {\bm{\gamma}}^* = \bm {l} -\bm {g}^*,
\end{equation}
where ${A}_{ij} = a(v_i, v_j)$, ${l}_{i} = \ell(v_i)$ for $i,j = 1,\ldots,J$, and $\bm {g}^* \sim \N(\bar{\bm{g}}^*, \mathrm{cov}(\bm{g}^*))$ with
\begin{equation}
\label{gbarcovg}
\bar{\bm{g}}^* = \bm{K}(\Phi^*,\Phi) \bm{\beta},  \qquad \mathrm{cov}(\bm{g}^*) = \bm{K}(\Phi^*,\Phi^*) - \bm{K}(\Phi^*,\Phi)  \bm{D}^{-1}\bm{K}(\Phi,\Phi^*) ,
\end{equation}
for $\Phi^* \equiv [v_1,v_2,\ldots,v_J]$. It thus follows from (\ref{linearsystem}) that
\begin{equation}
\label{gammadis}
\bar{{\bm{\gamma}}}^* = \bm{A}^{-1}(\bm{l} - \bar{\bm{g}}^*), \qquad \mathrm{cov}( \bm{\gamma}^*) = \bm{A}^{-1} \mathrm{cov}(\bm{g}^*) \bm{A}^{-T} ,
\end{equation}
as $\bm g^*$ is Gaussian and $\bm A$ is invertible. 

Now let $\bm x_i \in \Omega, i = 1,\ldots, N$ be spatial points at which we would like to evaluate the predictive mean and covariance of $u^*$. Let $\bm{V} \in \mathbb{R}^{N \times J}$ be a matrix with entries $V_{ij} = v_j( \bm x_i), i=1,\ldots,N, j = 1,\ldots,J$.  It then follows from (\ref{ustarsum}) that 
\begin{equation}
\label{ustar}
\bm{u}^* = \bm V \bm \gamma^* ,
\end{equation}
where $u^*_i = u^*(\bm x_i)$, $i = 1,\ldots,N$.  It follows from (\ref{gammadis}) and (\ref{ustar}) that
\begin{equation}
\bm u^* \sim \N(\bar{\bm{u}}^*, \mathrm{cov}({\bm{u}}^*)), 
\end{equation}
where
\begin{equation}
\label{ustardist}
\bar{\bm{u}}^*  =   \bm U (\bm{l} - \bar{\bm{g}}^*), \qquad \mathrm{cov}({\bm{u}}^*)  =  \bm U \mathrm{cov}(\bm{g}^*) \bm U^T ,
\end{equation}  
with $\bm U = \bm{V} \bm{A}^{-1}$. We examine the posterior distribution as given by (\ref{ustardist}). Note first that  the posterior mean  $\bar{\bm{u}}^*$ is the difference between two terms: the first term $\bm u = \bm U \bm l$ is simply the best knowledge state $u$ evaluated at $\bm x_i, i = 1,\ldots, N$; the second term $\bm U \bar{\bm g}^*$ is a correction term to the best knowledge state and is obtained by using our functional Gaussian process regression. Note also that the posterior covariance is a quadratic form of $\bm U$ with the posterior covariance matrix $\mathrm{cov}(\bm{g}^*)$, showing that the predictive uncertainty grows with the magnitude of $\bm U$.  Hence, the predictive uncertainty depends on the inverse matrix $\bm A^{-1}$. The implementation of our method for computing the posterior distribution (\ref{ustardist}) is shown in Figure 1.

\begin{figure}[htbp]
\centering
\framebox{\makebox[\width][c]{
$ \begin{array}{l}
\mbox{\textbf{Input}:}  \quad a, \ell, \{c_i\}_{i = 1}^M \ (\mbox{best knowledge}), \{v_j\}_{j=1}^J (\mbox{basis functions}),  \ \bm d \ (\mbox{observed data}),  \\[1ex]
\qquad \qquad \sigma \ (\mbox{noise level}), \ k \ (\mbox{covariance operator}), \ \{\bm{x}_i \in \Omega\}_{i=1}^N \ (\mbox{spatial coordinates}) \\[2ex]
\mbox{1. \hspace{5 pt} Compute }  \bm s \mbox{ by solving the best knowledge model (\ref{eq6})}  \\[1.5ex]
\mbox{2. \hspace{5 pt} Compute the adjoint states } \{\phi_i\}_{i=1}^M \mbox{ by solving~(\ref{adjoint})} \\[1.5ex]
\mbox{3. \hspace{5 pt} Compute } (\bar{\bm{g}}^*, \mathrm{cov}(\bm{g}^*)) \mbox{ in (\ref{gbarcovg}) by using functional Gaussian process regression} \\[1.5ex] 
\mbox{4. \hspace{5 pt} Form } \bm{A}, \bm{l}, \bm{V} \mbox{ and solve } \bm A^T \bm U^T = \bm V^T \mbox{ to obtain } \bm {U} \\[1.5ex] 
\mbox{5. \hspace{5 pt} Compute } \bar{\bm{u}}^*  =   \bm U (\bm{l} - \bar{\bm{g}}^*), \ \mathrm{cov}({\bm{u}}^*)  =  \bm U \mathrm{cov}(\bm{g}^*) \bm U^T \\[1.5ex] 
\mbox{\textbf{Return}:}  \quad \bar{\bm{u}}^* \ (\mbox{posterior mean}), \ \mathrm{cov}({\bm{u}}^*) \ (\mbox{posterior covariance}) 
\end{array} $
}}
\caption{Main algorithm for computing the posterior distribution of $\bm u^*$.}
\label{fig:1-0}
\end{figure}



\subsection{Relationship with least-squares regression}

Here, we show an alternative approach to computing the posterior mean in (\ref{ustardist}) by solving a deterministic least-squares problem. We note that the posterior mean $\bar{u}^* \in V$ satisfies 
\begin{equation}
\label{meanweakform}
a(\bar{u}^*, v)    = \ell(v) - \bar{g}^*(v),  \quad \forall v \in V .
\end{equation}
Here, for any $v \in V$, $\bar{g}^*(v)$ is the posterior mean of $g(v)$ and given by
\begin{equation}
\label{meang}
\bar{g}^*(v) = \sum_{j=1}^M \beta_j k(\phi_j, v) , \quad \forall v \in V ,
\end{equation}
where  the adjoint states  $\phi_j, j = 1,\ldots, M$ satisfy  (\ref{adjoint}) and the coefficient vector $\bm \beta$ is the solution of (\ref{betaD}) . It thus follows that the mean state $\bar{u}^* \in V$ satisfies
\begin{equation}
\label{meanweakform2}
a(\bar{u}^*, v)    = \ell(v) -  k(\bar{q}^*, v),  \quad \forall v \in V ,
\end{equation}
where $\bar{q}^* = \sum_{i=1}^M \beta_i \phi_i$ is the weighted sum of the adjoint states. We then evaluate the mean outputs of the stochastic PDE model as
\begin{equation}
\bar{s}^*_i = c_i(\bar{u}^*), \qquad i = 1,\ldots M.
\end{equation}
The following lemma sheds light on the relationship between the mean outputs and the observed data.

\begin{lem}
Assume that the covariance operator $k(\cdot,\cdot)$ is a bilinear form.  We have that $\bar{\bm s}^* = \bm{d} - \sigma \bm \beta$.
\end{lem}
\begin{proof}
We first note from the adjoint equation (\ref{adjoint}) and (\ref{meanweakform2}) that
\begin{equation}
\bar{s}^*_i = c_i(\bar{u}^*) = - a(\bar{u}^*, \phi_i) = \sum_{j=1}^M  k(\phi_j, \phi_i) \beta_j - \ell(\phi_i), \quad i = 1,\ldots, M  .
\end{equation}
We next recall that $\bm \beta$ satisfies 
\begin{equation}
\sum_{j=1}^M (k(\phi_j, \phi_i) + \sigma \delta_{ij}) \beta_j = d_i - s_i, \quad i = 1,\ldots, M ,
\end{equation}
where  $\delta_{ij}$ is the Kronecker delta. Moreover, we obtain from the best knowledge model (\ref{eq6}) and the adjoint equation (\ref{adjoint}) that
\begin{equation}
\ell(\phi_i) =  a(u, \phi_i) = - c_i(u) = - s_i, \quad i = 1,\ldots, M    .
\end{equation}
The desired result immediately follows from the above three equations. This completes the proof. 
\end{proof}

This lemma shows that the mean outputs differs from the observed data by the product of the noise level $\sigma$ and the coefficient vector $\bm \beta$. When the observed data is noise-free (namely, $\sigma = 0$) we have that the mean output vector is exactly equal to the observed data. Henceforth,  whenever $M$ is sufficiently large and $\sigma$ is relatively small, we expect that our method will yield a much better estimate of the true state than the best knowledge model. The following theorem shows the optimality of the mean state.

\begin{thm}
Assume that the covariance operator $k(\cdot,\cdot)$ is a bilinear form.  Then we have $(\bar{u}^*,\bar{q}^*,\bm \beta) = (u^{\rm o},q^{\rm o},\bm \beta^{\rm o})$, where
\begin{equation}
\label{eq8}
\begin{split}
(u^{\rm o},q^{\rm o},\bm \beta^{\rm o}) = & \arg \min_{z \in V, q \in V, \bm \gamma \in \mathbb{R}^M} \frac{1}{2}k(q,q) + \frac{1}{2} \sigma \bm \gamma^T \bm \gamma \\
& \qquad \ \mathrm{s.t. } \ \  a(z, v) + k(q,v)  = \ell(v),  \quad \forall v \in V, \\
& \qquad \quad \quad \ c_i(z) + \sigma \gamma_i = d_i, \quad i = 1,\ldots, M \ .
\end{split}
\end{equation}
\end{thm}
\begin{proof}
We introduce the Lagrangian 
\begin{equation}
\mathcal{L}(q,z,\bm \gamma, p,\bm \varrho) = \frac{1}{2}k(q,q) + \frac{1}{2} \sigma \bm \gamma^T \bm \gamma - a(z, p) - k(q,p)  + \ell(p) - \sum_{i=1}^M \varrho_i (c_i(z) + \sigma \gamma_i - d_i),
\end{equation}
where $p \in V$ and $\bm \varrho \in \mathbb{R}^M$ are the Lagrange multipliers of the constraints. The optimal solution $(q^{\rm o},u^{\rm o},\bm \beta^{\rm o}, p^{\rm o},\bm \varrho^{\rm o})$ satisfies
\begin{equation}
\frac{\partial \mathcal{L}}{\partial q} = \frac{\partial \mathcal{L}}{\partial z} =  \frac{\partial \mathcal{L}}{\partial \bm \gamma}  = \frac{\partial \mathcal{L}}{\partial p} =  \frac{\partial \mathcal{L}}{\partial \bm \varrho} = 0,
\end{equation}
which yields
\begin{subequations}
\label{KKT}
\begin{alignat}{2}
k(q^{\rm o},v) - k( p^{\rm o},v)  & = 0,  \quad \forall v \in V, \\[1ex]
a(v, p^{\rm o})  + \sum_{i=1}^M \varrho^{\rm o}_i c_i(v)  & = 0,  \quad \forall v \in V, \\[1ex]
\sigma \beta_i^{\rm o} - \sigma \varrho^{\rm o}_i & = 0 , \quad i = 1,\ldots, M,\\[1ex]
a(u^{\rm o}, v) + k(q^{\rm o},v) - \ell(v) & = 0,  \quad \forall v \in V, \\[1ex]
c_i(u^{\rm o}) + \sigma \beta_i^{\rm o} & = d_i , \quad i = 1,\ldots, M\ .
\end{alignat}
\end{subequations}
Note that  when taking the partial derivatives we have used the assumption that $k$ is bilinear. The first two equations of (\ref{KKT}) yield that
\begin{equation}
\label{proof46}
q^{\rm o} = p^{\rm o} = \sum_{i=1}^M \varrho^{\rm o}_i \phi_i ,
\end{equation}
where $\phi_i, i = 1,\ldots,M$ are the adjoint states. And the third equation (\ref{KKT}c) gives
\begin{equation}
\label{betarho}
\beta_i^{\rm o} = \varrho^{\rm o}_i,\quad i = 1,\ldots, M . 
\end{equation}
Therefore, if we can show that $\bm \beta^{\rm o} = \bm \beta$ then (\ref{KKT}d) and (\ref{proof46}) imply that $(\bar{u}^*,\bar{q}^*) = (u^{\rm o},q^{\rm o})$. To this end, we note from (\ref{KKT}d) and the adjoint equation (\ref{adjoint}) that
\begin{equation}
\label{proof47}
c_i(u^{\rm o}) = - a(u^{\rm o}, \phi_i) = k(q^{\rm o}, \phi_i) - \ell(\phi_i), \quad i = 1,\ldots, M  .
\end{equation}
Moreover, we obtain from the best knowledge model (\ref{eq6}) and the adjoint equation (\ref{adjoint}) that
\begin{equation}
\label{proof48}
\ell(\phi_i) =  a(u, \phi_i) = - c_i(u) = - s_i, \quad i = 1,\ldots, M    .
\end{equation}
Finally, it follows from (\ref{KKT}e), (\ref{proof46}), (\ref{betarho}), (\ref{proof47}), and (\ref{proof48})  that
\begin{equation}
 \sum_{j=1}^m  \left( k(\phi_i,\phi_j) + \sigma \delta_{ij} \right) \beta^{\rm o}_j = d_i - s_i, \quad i = 1,\ldots,M  ,
\end{equation}
which implies that $\bm \beta^{\rm o} = \bm \beta$. This completes the proof.
\end{proof}

This theorem establishes a connection between functional GP regression and traditional least-squares regression when the covariance operator is bilinear. In particular, the posterior mean state $\bar{u}^*$ is the optimal solution of a least-squares minimization. This is hardly a surprise as the posterior mean  state is also the maximum a posteriori (MAP) estimate of linear functional regression model in the Bayesian framework discussed in the Appendix A. It is well known that the MAP estimate coincides with the least-squares solution. The main advantage of our approach over least-squares regression is that we can compute not only the posterior mean state but also the posterior covariance. Another advantage of our approach is that it allows us to choose a covariance operator based on the observed data by exploiting the marginal likelihood function, whereas least-squares regression does not provide a mechanism to optimally set the prior covariance operator.

\section{A Simple Heat Conduction Example}

\subsection{Problem description}

For the true PDE model we consider a one-dimensional heat equation:
\begin{equation}
\label{heatequation}
-\frac{\partial^2 u^\true}{\partial x^2} = f^{\true},  \quad \mbox{in } \Omega^{\true} \equiv (-1,1),
\end{equation}
with Dirichlet boundary conditions $u^\true(-1) = u^\true(1) = 0$. The function space $V^{\true}(\Omega^{\true})$ is then given by
\begin{equation}
V^{\true}(\Omega^{\true}) = \left\{v \ : \ \int_{\Omega^\true} \left(v^2 + \frac{\partial v}{\partial x} \frac{\partial v}{\partial x} \right) dx < \infty  \mbox{ and } v(-1) = v(1) = 0\right\} .
\end{equation}
The true state $u^\true \in V^{\true}(\Omega^{\true})$ satisfies 
\begin{equation}
a^\true(u^\true, v) = \ell^\true(v), \quad \forall v \in V^{\true}(\Omega^{\true}),
\end{equation}
where
\begin{equation}
a^\true(w,v) = \int_{\Omega^\true} \frac{\partial w}{\partial x} \frac{\partial v}{\partial x} dx, \quad \ell^\true(v) =  \int_{\Omega^\true} f^{\true} v dx, \quad \forall w,v \in V^{\true}(\Omega^{\true}) .
\end{equation}
We prescribe a {\em synthetic source term}  as $f^\true = \sin(\pi x) + 4 \sin(4\pi x)$. It is easy to see that
\begin{equation}
u^\true =  \frac{\sin(\pi x)}{\pi^2}  + \frac{\sin(4\pi x)}{4\pi^2}  .
\end{equation}
The true state is unknown to us and will serve to assess the performance of our method.

We next assume that we know almost everything about the true model except for the source term $f^\true$ and the boundary data. 
We introduce a function space $V(\Omega)$ with $\Omega = (-1,1)$ as
\begin{equation}
\label{spaceV}
V(\Omega) = \left\{v \ : \ \int_{\Omega} \left(v^2 + \frac{\partial v}{\partial x} \frac{\partial v}{\partial x} \right) dx < \infty  \mbox{ and } v(-1) = b_1, v(1) = b_2\right\} ,
\end{equation}
where the boundary data $b_1$ and $b_2$ will be determined from the observed data. We then define our best knowledge model: find $u \in V(\Omega)$ such that 
\begin{equation}
\label{best knowledgemodel}
a(u, v) = \ell(v), \quad \forall v \in V(\Omega),
\end{equation}
where
\begin{equation}
\label{best knowledgemodel2}
a(w,v) = \int_{\Omega} \frac{\partial w}{\partial x} \frac{\partial v}{\partial x} dx, \qquad \ell(v)  = \int_{\Omega} f v dx, \quad \forall w,v \in V(\Omega) ,
\end{equation}
In practice, we replace the continuous space $V(\Omega)$ with a discrete counterpart, for this problem, a 2000-element linear finite element space.

\subsection{Model specifications}

We now specify the observation functionals $c_i(v) = \int_{\Omega} \delta(x_i) v d x = v(x_i), i = 1,\ldots,M$, where $\delta(x)$ is the Dirac delta function and the $x_i$ are the extended Chebyshev nodes [refs] in the interval $[-1,1]$:
\begin{equation}
x_i = -\frac{\cos((2i-1)\pi/(2M))}{\cos(\pi/(2M))} , \quad i = 1,\ldots, M.  
\end{equation}
These functionals correspond to pointwise observations taken at the points $x_i$. Note that the set of measurement points $\{x_i\}_{i=1}^M$ varies with $M$. Note also that the finite element mesh is designed to include $\{x_i\}_{i=1}^M$ in its grid points.

We shall assume that the observations are noise-free, that is we have  $\sigma = 0$ and $d_i = u^\true(x_i), i = 1,\ldots,M$. Since the observations at $x_1$ and $x_{M}$ are used to define the function space $V(\Omega)$ in (\ref{spaceV}) (that is we set $b_1 = d_1 = 0$ and $b_2 = d_M = 0$), we can only use the remaining $(M-2)$ observations to construct the training set as
\begin{equation}
\label{TII}
T = \{(\phi_i,d_i - s_{i}), \ i = 2,\ldots,M-1\} ,
 \end{equation}
where $\phi_i \in V(\Omega), i = 2,\ldots,M-1$ satisfies
\begin{equation}
a(v,\phi_i) = -c_i(v), \quad \forall v \in V(\Omega) ,
\end{equation}
and $s_{i} = u(x_i), i = 2,\ldots,M-1$ are the outputs of the best knowledge model. Hence, the training set has only $(M-2)$ samples.  Furthermore, we use a bilinear covariance operator of the form
\begin{equation}
k(w,v; \bm \theta) = \int_{\Omega} \left(\theta_1 wv + \theta_2 \frac{\partial w}{\partial x} \frac{\partial v}{\partial x} \right) dx ,
\end{equation}
The parameters $\bm \theta = (\theta_1,\theta_2)$ are determined by maximizing the log marginal likelihood (\ref{LML}).

We will compare our method to the standard Gaussian process regression described in Section 3 which ignores the best knowledge model and utilizes only the data. To this end, we employ a squared-exponential covariance function of the form
\begin{equation}
\kappa(x,x'; \bm \zeta) =  \zeta_1^2 \exp(- \frac{1}{2\zeta_2^2}(x-x')^2) ,
\end{equation}
where $\zeta_1$ represents the signal variance, while $\zeta_2$ represents the length scale. These parameters  are set by maximizing the log marginal likelihood (\ref{LMLGP}).

\subsection{Results and discussions}

\begin{figure}[thbp]
    \centering
        \includegraphics[scale=0.5]{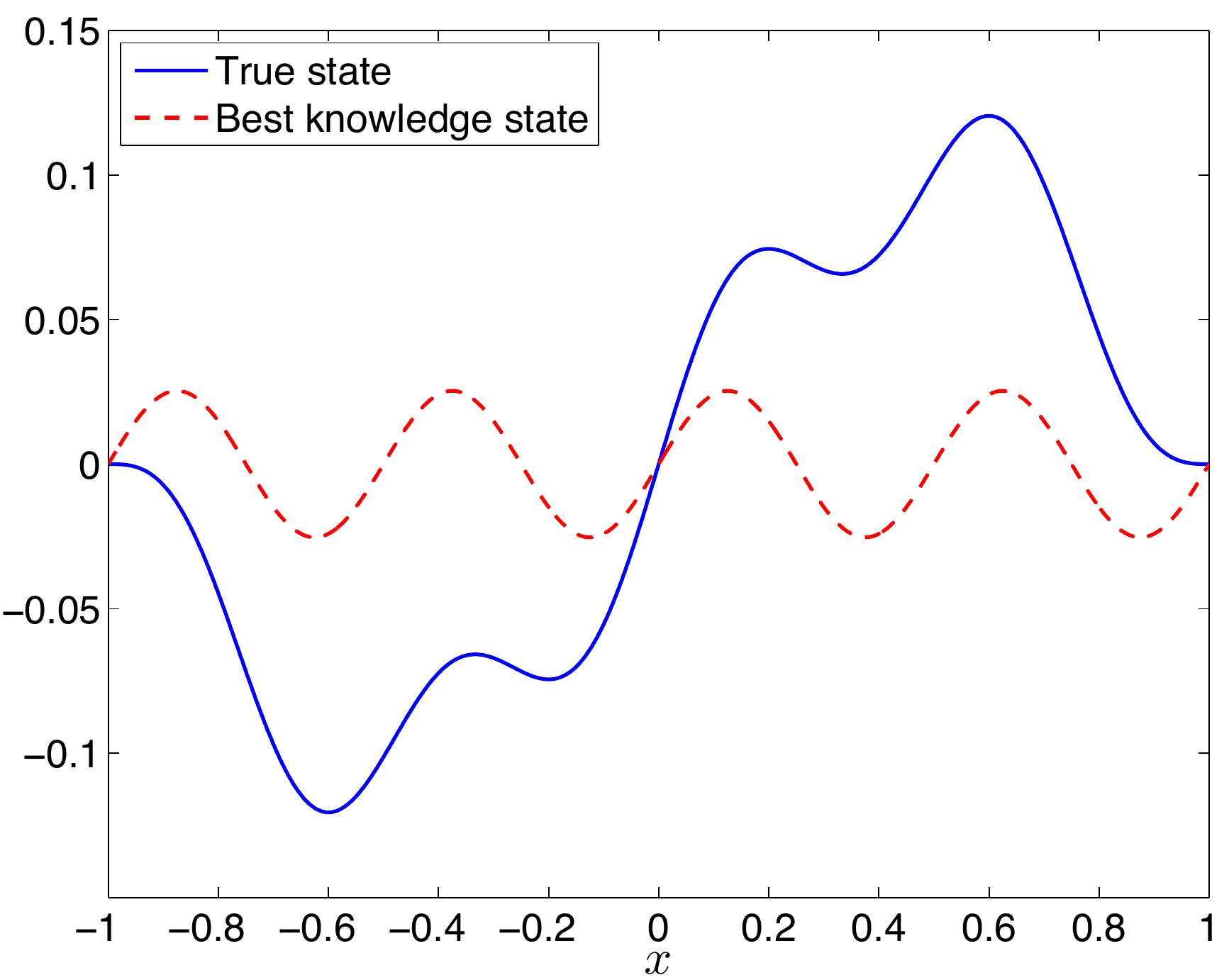}
    \caption{Plots of the true state and the best knowledge state.}
    \label{fig10}
\end{figure}

\begin{table}[hp]
  \begin{center}
\scalebox{0.95}{%
    $\begin{array}{|c||c|c|c||c|c|c|}
    \hline
&  \multicolumn{3}{|c||}{\mbox{Functional GP}} & \multicolumn{3}{|c|}{\mbox{Standard GP}}\\
\hline
M & (\theta_1,\theta_2) & ||u^\true - \bar{u}^*||_{\Omega} & ||\sqrt{\mbox{var}(u^*)}||_{\Omega} & (\zeta_1,\zeta_2)  & ||u^\true - \bar{h}^*||_{\Omega} & ||\sqrt{\mbox{var}(h^*)}||_{\Omega} \\
\hline
  4  &  (0.485 ,0) &  4.43\mbox{E-}3  &  4.64\mbox{E-}2  &  (0.053, 0.064)  &  9.14\mbox{E-}2  &  6.84\mbox{E-}2  \\  
  5  &  (0.319 ,0) &  6.02\mbox{E-}3  &  2.47\mbox{E-}2  &  (0.076, 0.005)  &  8.12\mbox{E-}2  &  9.73\mbox{E-}2  \\  
  6  &  (0.247 ,0) &  1.41\mbox{E-}3  &  1.60\mbox{E-}2  &  (0.062, 0.046)  &  8.68\mbox{E-}2  &  7.84\mbox{E-}2  \\  
  7  &  (0.199 ,0) &  9.06\mbox{E-}4  &  1.12\mbox{E-}2  &  (0.086, 0.621)  &  3.13\mbox{E-}2  &  4.40\mbox{E-}4  \\  
  8  &  (0.167 ,0) &  4.16\mbox{E-}4  &  8.31\mbox{E-}3  &  (0.074, 0.306)  &  3.64\mbox{E-}2  &  8.36\mbox{E-}3  \\  
  9  &  (0.143 ,0) &  2.38\mbox{E-}4  &  6.42\mbox{E-}3  &  (0.067, 0.203)  &  3.25\mbox{E-}2  &  2.05\mbox{E-}2  \\  
  10  &  (0.125 ,0) &  1.42\mbox{E-}4  &  5.12\mbox{E-}3  &  (0.064, 0.186)  &  9.39\mbox{E-}3  &  1.85\mbox{E-}2  \\  
  11  &  (0.111 ,0) &  9.10\mbox{E-}5  &  4.18\mbox{E-}3  &  (0.062, 0.223)  &  1.90\mbox{E-}2  &  7.00\mbox{E-}3  \\  
  12  &  (0.100 ,0) &  6.10\mbox{E-}5  &  3.48\mbox{E-}3  &  (0.065, 0.193)  &  9.24\mbox{E-}4  &  9.43\mbox{E-}3  \\  
  13  &  (0.091 ,0) &  4.25\mbox{E-}5  &  2.94\mbox{E-}3  &  (0.061, 0.196)  &  4.57\mbox{E-}3  &  5.96\mbox{E-}3  \\  
  14  &  (0.084 ,0) &  3.06\mbox{E-}5  &  2.52\mbox{E-}3  &  (0.064, 0.200)  &  4.24\mbox{E-}4  &  3.87\mbox{E-}3  \\  
  15  &  (0.077 ,0) &  2.26\mbox{E-}5  &  2.18\mbox{E-}3  &  (0.065, 0.205)  &  1.58\mbox{E-}3  &  2.22\mbox{E-}3  \\  
 \hline
 \end{array} $
}
\end{center}{$\phantom{|}$}
\caption{The optimal hyperparameters, the $L^2(\Omega)$ norm of the prediction error ($u^\true(x) - \bar{u}^*(x)$ in our method and $u^\true(x) - \bar{h}^*(x)$ in the standard GP regression), and the $L^2(\Omega)$ norm of the standard deviation function ($\sqrt{\mbox{var}({u}^*(x))}$ in our method and $\sqrt{\mbox{var}({h}^*(x))}$ in the standard GP regression) as a function of $M$ for both functional GP regression and standard GP regression. Here $\bar{h}^*(x)$ and $\mbox{var}({h}^*(x))$ are the mean prediction and the posterior variance of $u^\true(x)$ for the standard GP regression.}
   \label{tab3}
\end{table}

We consider $f = 4\sin(4\pi x)$ for the best knowledge model (\ref{best knowledgemodel})-(\ref{best knowledgemodel2}). This yields the best knowledge state $u = {\sin(4\pi x)}/({4\pi^2})$. Figure \ref{fig10} shows the true state $u^\true$ and the best knowledge state $u$. We observe that $u$ is considerably different from $u^\true$. Therefore, the best knowledge model does not produce a good prediction of the heat equation (\ref{heatequation}). We now apply functional GP regression to this example and present numerical results to demonstrate the performance of our method relative to the standard GP regression.

We present in Table \ref{tab3} the optimal hyperparameters, the $L^2(\Omega)$ norm of the prediction error,  and the $L^2(\Omega)$ norm of the posterior standard deviation (the square root of the posterior variance) for our method and the standard GP regression. Here the $L^2(\Omega)$ norm of a function $v$ is defined as $||v||_{\Omega} =  (\int_{\Omega} v^2 d x)^{1/2}$. We observe that while $\theta_2$ is always zero, $\theta_1$ decreases as $M$ increases, indicating that the prediction uncertainty is reduced as the number of observations increases. We also note that the length scale $\zeta_2$ of the squared-exponential covariance function is relatively small for $M \le 6$, indicating that the training set may be inadequate for the standard GP regression to produce a good prediction. 

We see from Table \ref{tab3}  that the prediction error in our method converges significantly faster than that in the standard GP regression as $M$ increases. Therefore, our method requires fewer observations to achieve the same accuracy. In particular,  our method with $4$ observations has slightly smaller error than the standard GP regression with $13$ observations. This is made possible because our method uses both the best knowledge model and the observations to do regression on the space of functionals, whereas the standard GP progression uses the observations only to do regression on the space of functions. We also observe that the posterior standard deviation (measured in $L^2(\Omega)$ norm) shrinks with increasing $M$ albeit at a slower rate than the prediction error, indicating that our posterior variance of the prediction error is rigorous. In contrast, the standard GP regression has the posterior standard deviation even smaller than the prediction error for small values of $M$, indicating that the posterior variance of the standard GP regression may not be rigorous when the training set is inadequate. This can be attributed to the fact the standard GP regression requires a large enough set of observations to provide accurate prediction and rigorous error estimation.

Finally, we show in Figure \ref{fig12} the true state, the mean prediction, and  the 95\% confidence region (shaded area) for our method (left panels) and the standard  GP regression (right panels). Here the 95\% confidence region is an area bounded by the mean prediction plus and minus two times the standard deviation function. Note that the prediction error is zero at the measurement points, which is consistent with the theoretical result stated in Lemma 1. Moreover, the standard deviation function is also zero at the measurement points --- a consequence of the fact that the prediction error is zero at those points. We see that our method does remarkably well even with just 4 observations when it is compared to the standard GP regression. For our method  the true state $u^\true$ resides in the 95\% confidence region which shrinks rapidly with as $M$ increases, whereas for the standard GP regression $u^\true$ does not always reside in the 95\% confidence region. Indeed, as seen in  Figure \ref{fig12}(d), the standard GP regression gives poor prediction and erroneous 95\% confidence region for $M = 8$. Although the standard GP regression provides more accurate prediction and rigorous 95\% confidence region for $M = 12$, it is still not as good as our method. In summary, the numerical results obtained for this simple example show that our method outperforms the standard GP regression.


\begin{figure}[thbp]
    \centering
        \includegraphics[scale=0.4]{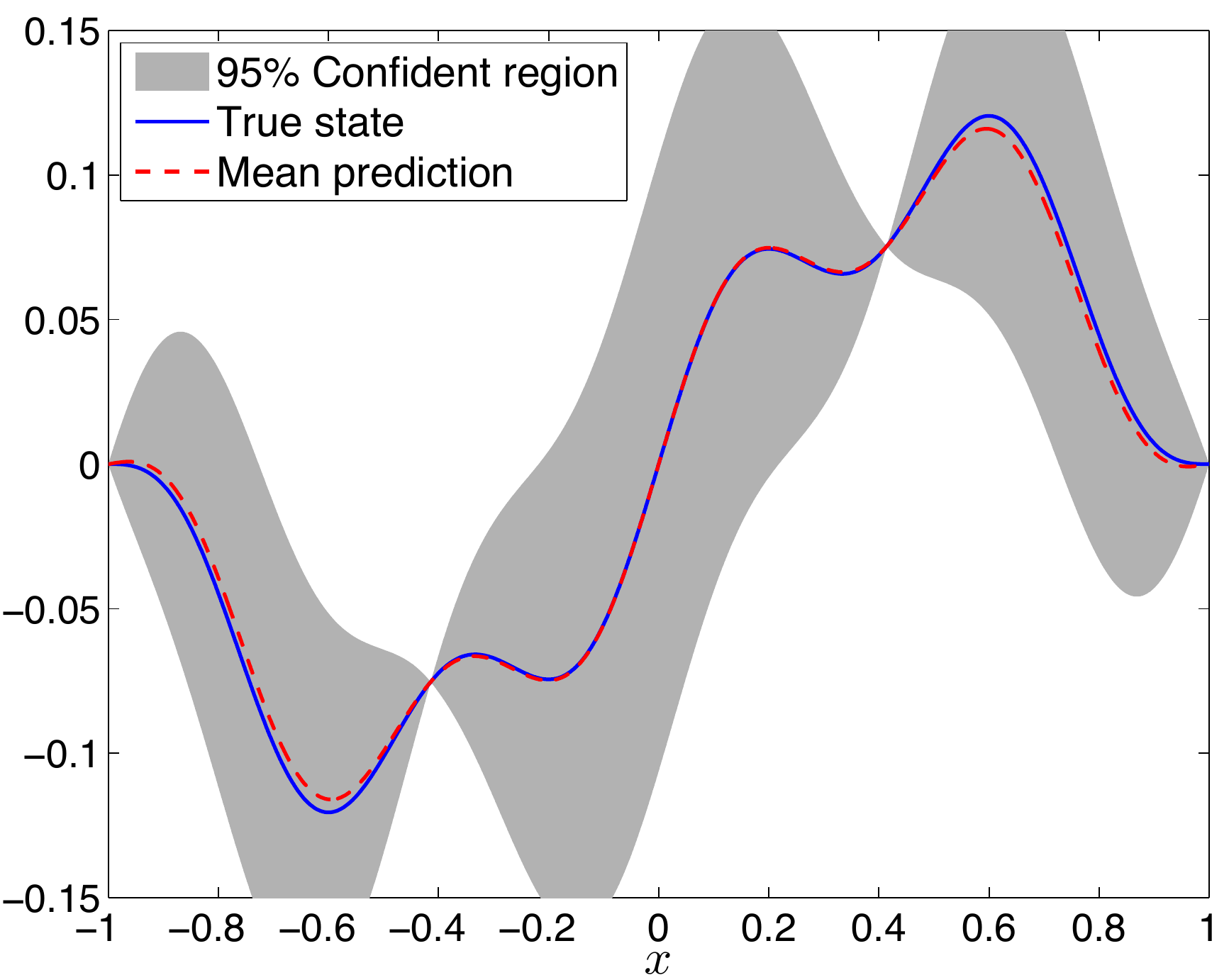} \qquad
	\includegraphics[scale=0.4]{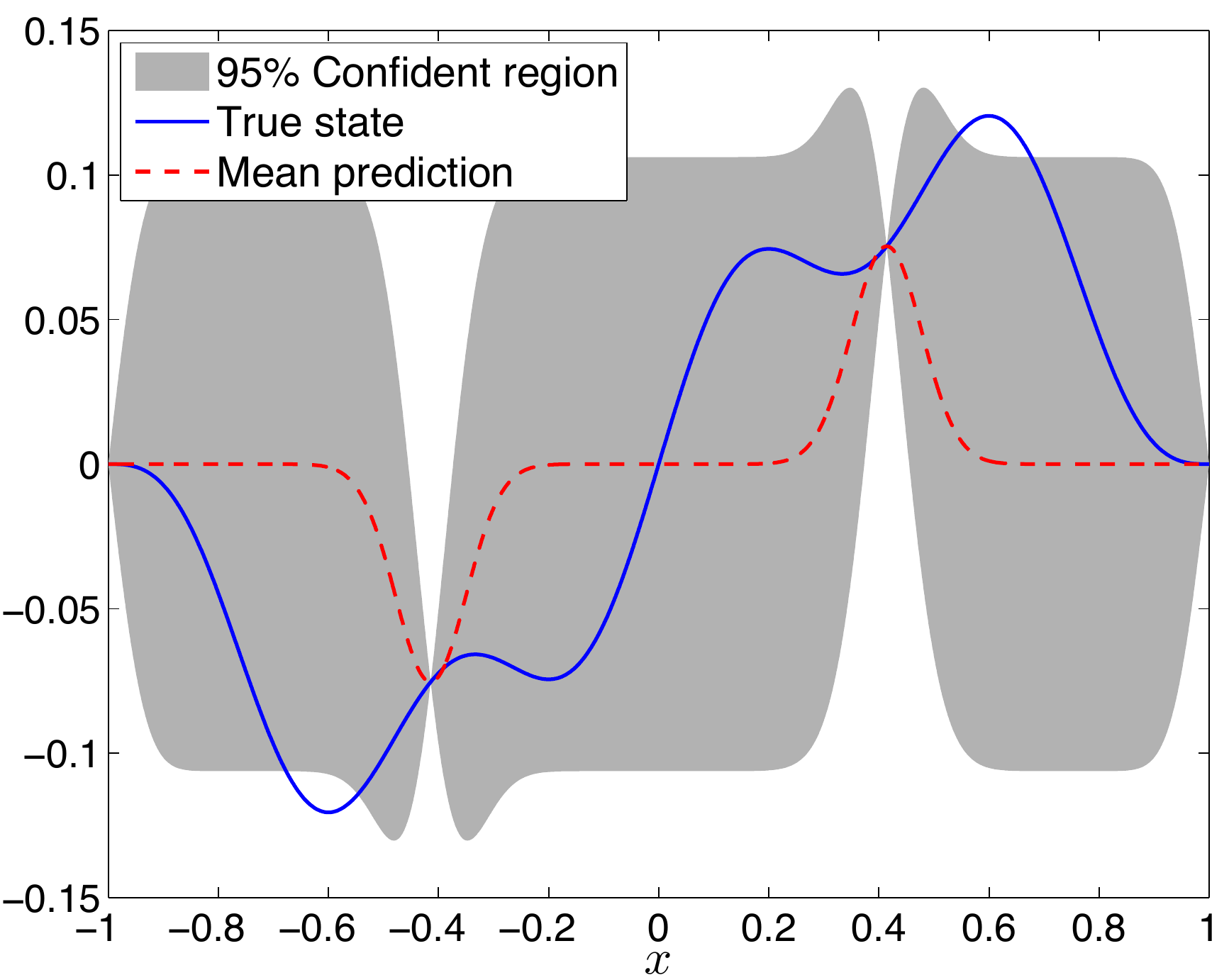} \\
(a) Functional GP for $M=4$ \hspace{2.5cm} (b) Standard GP for $M=4$  \\[1ex]
	        \includegraphics[scale=0.4]{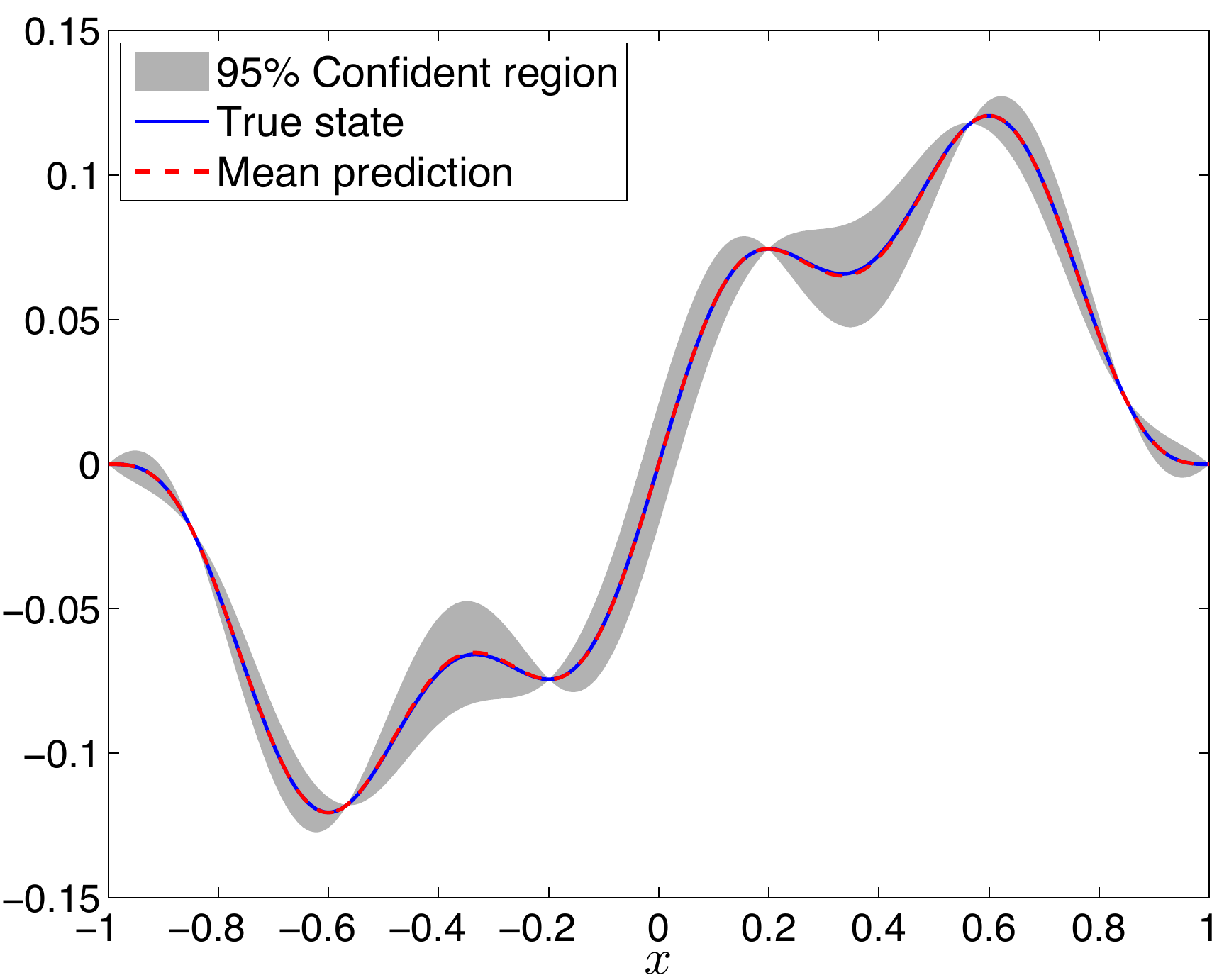} \qquad
	\includegraphics[scale=0.4]{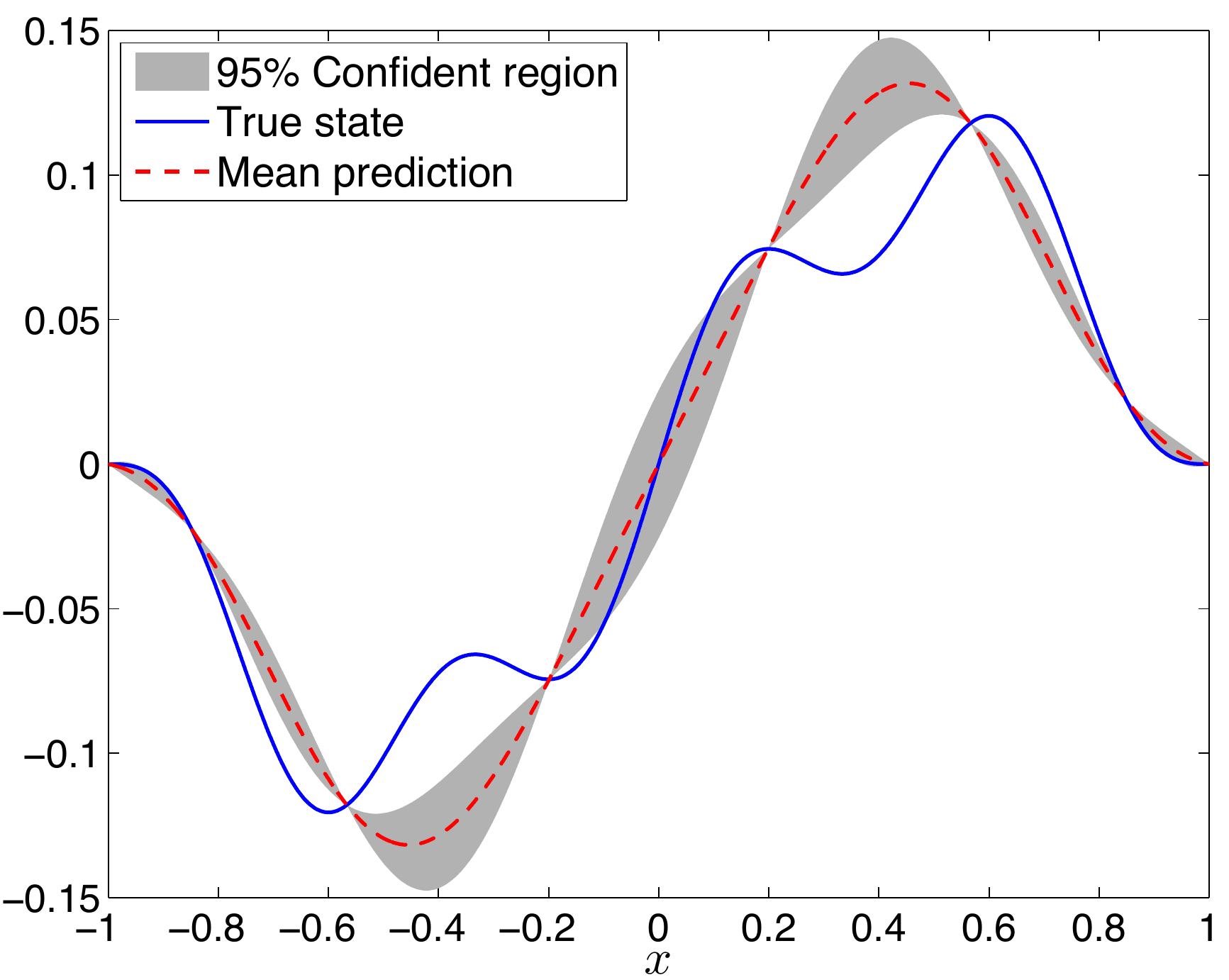} \\
(c) Functional GP for $M=8$ \hspace{2.5cm} (d) Standard GP for $M=8$  \\	[1ex]
	        \includegraphics[scale=0.4]{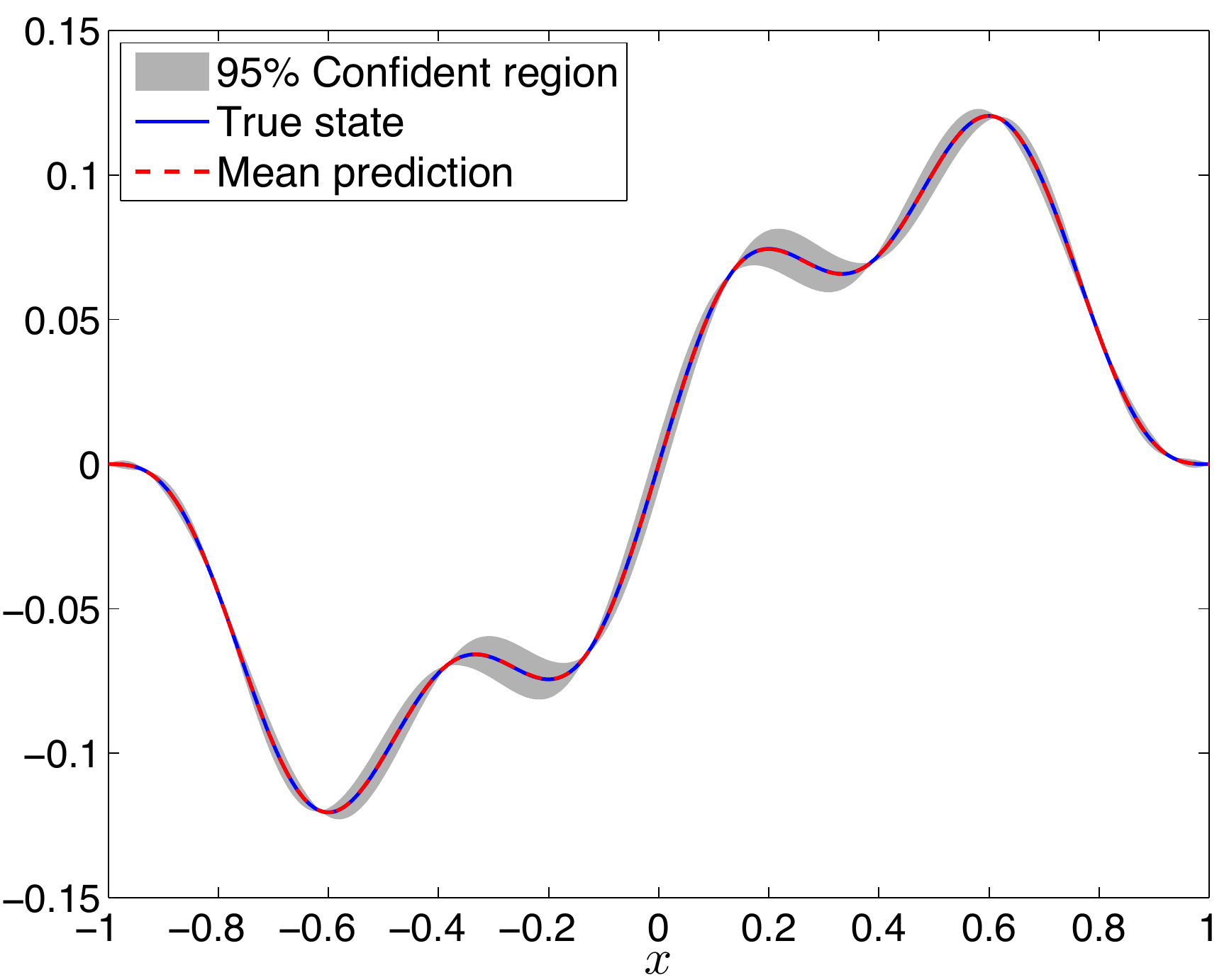} \qquad
	\includegraphics[scale=0.4]{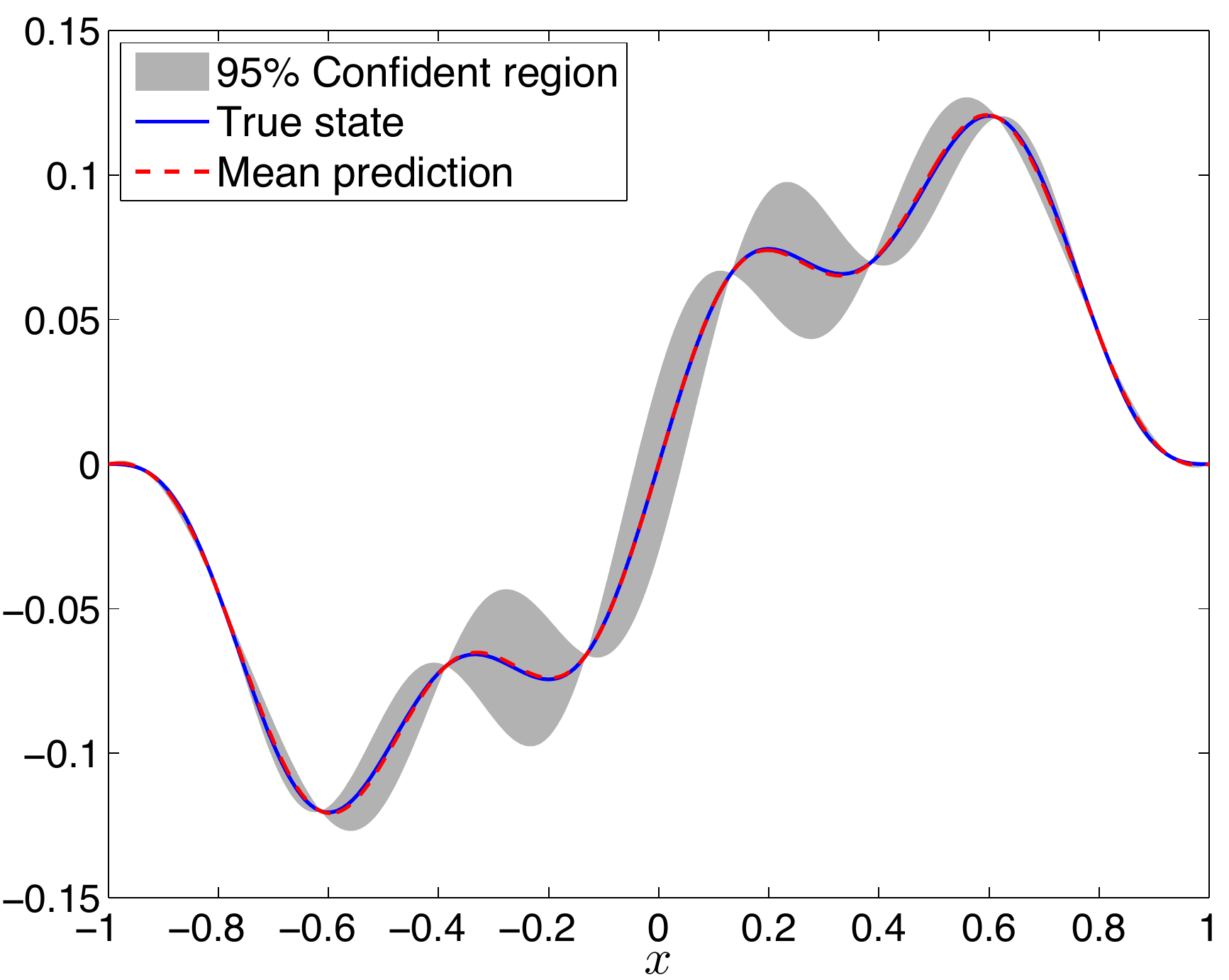} \\
(e) Functional GP for $M=12$ \hspace{2.5cm} (f) Standard GP for $M=12$  \\	
    \caption{Panels show the true state and the mean prediction obtained using $M = 4, 8, 12$ observations for both functional GP regression (left) and the standard GP regression (right). In these plots the shaded area represents the mean prediction plus and minus two times the standard deviation function  (corresponding to the 95\% confidence region).}
    \label{fig12}
\end{figure}

\section{Conclusions}

In this paper, we have presented a new statistical approach to the problem of combining the observed data with a mathematical model to improve our prediction of a physical system.  A new {\em functional Gaussian process} regression approach is presented, which has its root in Gaussian processes for functions. Our approach has the following unique properties. First, it allows for the incorporation of the best knowledge model into the regression procedure. Second, it can handle observations given in the form of linear functionals of the field variable.  Third, the method is non-parametric in the sense that it provides a systematic way to optimally determine the prior covariance based on the data. Fourth, our method can compute not only the mean prediction but also the posterior covariance, characterizing the uncertainty in our prediction of the physical state. These features distinguish our method from other regression methods. Numerical results were presented to highlight these features and demonstrate the superior performance of the proposed method relative to the standard GP regression.

We conclude the paper by pointing out several possible extensions and directions for further research. We would like to extend the proposed approach to nonlinear PDE models, which will broaden the application domain of our method. Nonlinear PDEs represent some significant challenges because the adjoint problems will depend on the true state which we do not know and because  our stochastic PDE does not preserve the Gaussian property of the functional $g$ due to nonlinearity.  We would also like to extend the method to goal-oriented statistical estimation in which we would like to infer new outputs rather the state of the physical system.

\section*{Acknowledgments}

We would like to thank Professor Robert M. Freund of Sloan School of Management at MIT for fruitful discussions. This work was supported by AFOSR Grant No. FA9550-11-1-0141, AFOSR Grant No. FA9550-12-0357, and the Singapore-MIT Alliance.

\appendix

\section{Bayesian linear functional regression}

In this section, we develop a nonparametric Bayesian framework for linear functional regression with Gaussian noise. We then show that this framework is equivalent to our functional Gaussian process regression described in Section 3. 

We begin by introducing an orthornormal basis set $\{\psi_i \in V\}_{i = 1}^J$ such that $\int_{\Omega} \psi_i \psi_j d \bm{x} = \delta_{ij}$, where  $\delta_{ij}$ is the Kronecker delta. We define associated linear functionals $\ell_i : V \to \mathbb{R}$ as
\begin{equation}
\ell_i(v)  = \int_{\Omega} \psi_i v d \bm{x} \equiv m(\psi_i,v), \quad i = 1,\ldots, J .
\end{equation}
The linear functional regression model is defined by
\begin{equation}
\label{LFRM}
g(v) = \sum_{i=1}^J w_i \ell_i(v) ,  \qquad y =  g(v) + \varepsilon ,
\end{equation}
where $v \in V$ is the input function, $\bm w = [w_1,\ldots,w_J]^T$ is a vector of weights (parameters) of the linear model, $g$ is the functional and $y \in \mathbb{R}$ is the observed target value. We assume that the observed value $y$ differs from the functional value $g(v)$ by additive noise
\begin{equation}
\varepsilon \sim \N(0,\sigma^2) .
\end{equation}
This noise assumption together with the model directly gives rise to the following likelihood
\begin{equation}
\label{appdix8}
p(\bm y | \Phi, \bm w ) = \prod_{j=1}^M  \frac{1}{\sqrt{2 \pi} \sigma} \exp \left( \frac{-(y_j - g(\phi_j))^2}{2 \sigma^2}\right) = \N(\bm L^T \bm w, \sigma^2),
\end{equation}
where $\bm L$ has entries $L_{ij} = \ell_i(\phi_j), i = 1,\ldots,J,j=1,\ldots,M$. Recall that $\phi_i, i = 1,\ldots,M$ are the adjoint states and that $y_i = d_i - s_i, i = 1,\ldots,M$ are the observed outputs which are the differences between the observed data and the best knowledge outputs.

In the Bayesian formalism we need to specify {\em a prior distribution} for $\bm w$, which encodes our belief about the weights prior to  taking into account the observations. We consider a zero-mean Gaussian prior distribution with covariance matrix $\bm \Lambda$ for the weights
\begin{equation}
\label{appdix9}
p(\bm w) = \N(0,\bm \Lambda) .
\end{equation}
Without loss of generality we assume that $\bm \Lambda$ is a diagonal matrix, that is, $\Lambda_{ij} = 0$ if $i \neq j$. If $\bm \Lambda$ is not a diagonal matrix then we can always introduce a new weight vector $\bm w' = \bm \Lambda^{-1/2} \bm w$, so that the covariance of $\bm w'$ is a diagonal matrix. Then we choose to work with $\bm w'$ instead of $\bm w$.  The posterior distribution over the weights is then obtained by using Bayes' rule as
\begin{equation}
\label{appdix10}
p(\bm w | \bm y , \Phi ) = \frac{p(\bm y | \Phi, \bm w ) p(\bm w)}{p(\bm y | \Phi)},
\end{equation}
where the normalizing constant is given by
\begin{equation}
p(\bm y | \Phi) = \int p(\bm y | \Phi, \bm w ) p(\bm w) d \bm w .
\end{equation}
Substituting (\ref{appdix8}) and (\ref{appdix9}) into (\ref{appdix10}) and working through some algebraic manipulations, we obtain
\begin{equation}
\label{pweights}
p(\bm w | \bm y , \Phi ) \sim \N( \frac{1}{\sigma^2} \bm B^{-1} \bm L \bm y ,\bm B^{-1}),
\end{equation}
where $\bm B = \sigma^{-2} \bm L \bm L^T + \bm \Lambda^{-1}$. The posterior distribution (\ref{pweights}) combines the likelihood and the prior distribution, and captures all information about the weights.

To make predictions for a test function $\phi^* \in V$ we average the linear functional regression model (\ref{LFRM}) over all possible values of the weights under their posterior distribution (\ref{pweights}). Hence, the predictive distribution for $g^* = g(\phi^*)$ at $\phi^*$ is given by 
\begin{equation}
\label{gdistribution}
g^* | \phi^*, \Phi, \bm y  \sim \N( \frac{1}{\sigma^2} \bm{l}(\phi^*)^T \bm B^{-1} \bm L \bm y , \bm{l}(\phi^*)^T \bm B^{-1} \bm{l}(\phi^*)),
\end{equation}
where $\bm{l}(\phi^*)$ has entries ${l}_i(\phi^*) = \ell_i(\phi^*), i = 1,\ldots,J$.  The predictive distribution is a Gaussian distribution as expected. However, the predictive distribution (\ref{gdistribution}) requires the matrix inversion $\bm B^{-1}$ which may be expensive. Fortunately, we can derive an equivalent distribution which is more efficient to compute than the original one (\ref{gdistribution}).

We recall the Woodbury, Sherman and Morrison (WSM) formula for the matrix inversion 
\begin{equation}
(\bm Z + \bm U \bm W \bm V^T)^{-1} = \bm Z^{-1} - \bm Z^{-1} \bm U (\bm W^{-1} + \bm V^T \bm Z^{-1} \bm U)^{-1} \bm V^T \bm Z^{-1} .
\end{equation}
Using the WSM formula with $\bm Z = \bm{\Lambda}, \bm W = \sigma^{-2} \bm I, \bm U = \bm V = \bm L$ we  obtain
\begin{equation}
\label{Binverse}
\bm B^{-1} = (\sigma^{-2} \bm L \bm L^T +  \bm \Lambda^{-1})^{-1} = \bm \Lambda - \bm \Lambda \bm L (\sigma^2 \bm I + \bm L^T \bm \Lambda \bm L)^{-1} \bm L^T \bm \Lambda \ .
\end{equation}
It thus follows that
\begin{equation}
\label{BinverseL}
\begin{split}
\bm B^{-1} \bm L & =  \bm \Lambda \bm L - \bm \Lambda \bm L (\sigma^2 \bm I + \bm L^T \bm \Lambda \bm L)^{-1} \bm L^T \bm \Lambda \bm L \\
& = \bm \Lambda \bm L (\sigma^2\bm I + \bm L^T \bm \Lambda \bm L)^{-1} ( (\sigma^2 \bm I + \bm L^T \bm \Lambda \bm L) -  \bm L^T \bm \Lambda \bm L) \\
& = \sigma^2 \bm \Lambda \bm L (\sigma^2\bm I + \bm L^T \bm \Lambda \bm L)^{-1} .
\end{split}
\end{equation}
Inserting (\ref{Binverse}) and (\ref{BinverseL}) into (\ref{gdistribution}) we get
\begin{equation}
\label{newgstar}
g^* | \phi^*, \Phi, \bm y  \sim \N(\bar{g}^*, \mathrm{var}({g^*})),
\end{equation}
where
\begin{equation}
\label{bargstar}
\begin{split}
\bar{g}^* & = \bm{l}(\phi^*)^T \bm \Lambda \bm L (\sigma^2 \bm I  + \bm L^T \bm \Lambda \bm L)^{-1}  \bm y , \\
\mathrm{var}({g^*}) & = \bm{l}(\phi^*)^T  \bm \Lambda  \bm{l}(\phi^*)) -  \bm{l}(\phi^*)^T  \bm \Lambda \bm L (\sigma^2 \bm I + \bm L^T \bm \Lambda \bm L)^{-1} \bm L^T \bm \Lambda  \bm{l}(\phi^*) 
\end{split}
\end{equation}
In the predictive distribution (\ref{newgstar}), we need to invert the matrix $(\sigma^2 \bm I + \bm L^T \bm \Lambda \bm L)$ of size $M \times M$. Hence, it is more attractive than the original distribution (\ref{gdistribution}).

Thus far, we have not mentioned how we construct the basis functions $\{\psi_i\}_{i=1}^J$ and the covariance matrix $\bm \Lambda$.
We observe that in the predictive distribution (\ref{newgstar}) we need to compute $\bm L^T \bm \Lambda \bm L, \bm{l}(\phi^*)^T \bm \Lambda \bm L, \bm{l}(\phi^*)^T  \bm \Lambda  \bm{l}(\phi^*))$, where both $\bm L$ and $\bm{l}(\phi^*)$ depend on the basis set $\{\psi_i\}_{i=1}^J$ through the definition of the linear functionals $\ell_i(\cdot), i = 1,\ldots,J$. The entries of these quantities have the form $\bm l (v)^T  \bm \Lambda  \bm{l}(v')$ for $v, v' \in V$. We can thus define 
\begin{equation}
k(v',v) = \bm l (v')^T  \bm \Lambda  \bm{l}(v)),  \quad \mbox{for all } v, v' \in V \ .
\end{equation}
It follows that
\begin{equation}
\label{kernelbasis}
k(v',v) = \sum_{i=1}^J \sum_{j=1}^J \ell_i(v') \Lambda_{ij} \ell_j(v) = \sum_{i=1}^J  m(\psi_i,v') \Lambda_{ii} m(\psi_i,v) , \quad \forall v',v  \in V,
\end{equation}
since $\ell_i(v) = m(\psi_i,v)$ and $\bm \Lambda$ is a diagonal matrix. We next choose $v' = \psi_n$ and invoke $m(\psi_i,\psi_n) = \delta_{in}$ to arrive at the following eigenvalue problem:
\begin{equation}
\label{eigenvalue}
k(\psi_{n},v) = \Lambda_{nn} m(\psi_n,v), \quad \forall v \in V .
\end{equation}
This equation shows that $\{\psi_n\}_{n=1}^J$ and $\{\Lambda_{nn}\}_{n=1}^J$ are the eigenfunctions and eigenvalues of the covariance operator $k(\cdot,\cdot)$. Once the covariance operator $k(\cdot,\cdot)$ is specified, we can construct the basis set  $\{\psi_n\}_{n=1}^J$ and the covariance matrix $\bm \Lambda$ by solving the eigenvalue problem (\ref{eigenvalue}).\footnote{Of course, we can also define the covariance operator $k(v',v)$ in terms of the basis set $\{\psi_n\}_{n=1}^J$ and the covariance matrix $\bm \Lambda$ by using the expression (\ref{kernelbasis}). However, this series expansion of the covariance operator render the inner products too computationally expensive and should be avoided. Instead, one should work directly with an explicitly analytical form of the covariance operator as listed in Table 2.} We can then make predictions using the equations (\ref{newgstar}) and (\ref{bargstar}). The predictive distribution (\ref{newgstar}) is nothing but the same as that of functional Gaussian process regression described in Section 3. However, unlike nonparametric Bayesian inference, functional Gaussian process regression does not require the computation of  eigenfunctions and eigenvalues. This is because functional Gaussian process regression needs only to compute the inner products $k(v,v')$ for some pairs $v, v' \in V$, which is known as the {\em kernel trick} \cite{RasWil06}.


\bibliographystyle{plain}
\bibliography{library.bib}

%
%
%

\end{document}